\documentclass[11pt,reqno]{amsart}

\usepackage{latexsym}
\usepackage{amsfonts}
\usepackage{amsmath}
\usepackage{graphicx}
\usepackage{color}
\usepackage{url}

\usepackage{amssymb,amsmath}

\def\bl{\begin{lemma}}
\def\el{\end{lemma}}
\def\bth{\begin{theorem}}
\def\eth{\end{theorem}}
\def\bc{\begin{corollary}}
\def\ec{\end{corollary}}
\def\bcj{\begin{conjecture}}
\def\ecj{\end{conjecture}}
\def\bpr{\begin{proposition}}
\def\epr{\end{proposition}}
\def\bde{\begin{definition}}
\def\ede{\end{definition}}
\def\E{\mathbb{E}}

\newcommand{\be}{\begin{eqnarray}}
\newcommand{\ee}{\end{eqnarray}}
\newcommand{\eps}{{\mbox{$\epsilon$}}}
\newcommand{\R}{{\mathbb R}}

\newcommand{\Z}{{\mathbb Z}}
\newcommand{\N}{{\mathbb N}}

\newcommand{\with}{\hbox{ {\rm with} }}
\renewcommand{\and}{\hbox{ {\rm and} }}
\newcommand{\len}{\hbox{{\rm len}}}
\newcommand{\off}{\hbox{ {\rm off} }}
\newcommand{\on}{\hbox{ {\rm only on} }}
\newcommand{\inn}{\hbox{ {\rm in} }}

\newcommand{\C}{{\mathcal{C}}}
\newcommand{\prob}{\mbox{\bf P}}
\newcommand{\p}{\mbox{\bf p}}

\newcommand{\lr}{\leftrightarrow}
\newcommand{\lrr}{\stackrel{\,\, r}{\leftrightarrow}}
\newcommand{\lrrc}{\stackrel{\,\,r-K}{\longleftrightarrow}}
\newcommand{\lrrell}{\stackrel{\,\, \ell}{\leftrightarrow}}
\newcommand{\QQ}{\mathcal{Q}}

\newcommand{\lrrl}{\stackrel{\,\, \lambda r}{\leftrightarrow}}

\newcommand{\bcr}{B}

\newtheorem{theorem}{Theorem}[section]
\newtheorem{definition}{Definition}[section]
\newtheorem{lemma}[theorem]{Lemma}

\newtheorem{claim}[theorem]{Claim}

\newtheorem{corollary}[theorem]{Corollary}
\newtheorem{proposition}[theorem]{Proposition}
\newtheorem{conjecture}[theorem]{Conjecture}

\theoremstyle{definition}
\numberwithin{equation}{section}

\DeclareMathOperator{\Hit}{Hit}

\input epsf.sty

\gdef\SetFigFont#1#2#3#4#5{} 

\begin{document}
\title{The Alexander-Orbach conjecture holds in~high~dimensions}
\author{Gady Kozma} \author{Asaf Nachmias}

\begin{abstract} We examine the incipient infinite cluster (IIC) of
    critical percolation in regimes where mean-field behavior has been
    established, namely when the dimension $d$ is large enough or when $d>6$
    and the lattice is sufficiently spread out. We find that random walk on
    the IIC exhibits anomalous diffusion with the spectral dimension
    $d_s=\frac{4}{3}$, that is, $p_t(x,x)= t^{-2/3+o(1)}$. This establishes a
    conjecture of Alexander and Orbach \cite{AO}. En route we calculate the
    one-arm exponent with respect to the intrinsic distance.

\end{abstract}


\maketitle


\section{{\bf  Introduction}}

We study the behavior of the simple random walk on the incipient
infinite cluster (IIC) of critical percolation on $\Z^d$. The IIC
is a random infinite connected graph containing the origin which
can be thought of as a critical cluster conditioned to be infinite
(see formal definition in \S\ref{perc} and in particular
(\ref{iic})). The {\em spectral dimension} $d_s$ of an infinite
connected graph $G$ is defined by
$$ d_s = d_s(G) = -2 \lim_{n \to \infty} {\log \p_{2n}(x,x) \over \log n} \qquad\qquad \hbox{{\rm (if this limit exists)}}\, ,$$
where $x \in G$ and $\p_n(x,x)$ is the return probability of the
simple random walk on $G$ after $n$ steps (note that if the limit
exists, then it is independent of the choice of $x$). Alexander
and Orbach \cite{AO} conjectured that $d_s = 4/3$ for the IIC in
all dimensions $d>1$, but their basis for conjecturing this in low
dimensions was mostly rough correspondence with numerical results
and it is now believed that the conjecture is false when $d<6$
\cite[7.4]{Hu}. In this paper we establish their conjecture in
high dimensions.

\begin{theorem} \label{mainthm} Let $\prob_\mathrm{IIC}$ be the IIC
measure of critical percolation on $\Z^d$ with large $d$ ($d \geq 19$ suffices)
or with $d > 6$ and sufficiently spread-out lattice and consider the
simple random walk on the IIC. Then $\prob_\mathrm{IIC}$-a.s.
$$ \lim _{n \to \infty} {\log \p_{2n} (0,0) \over \log
n} = -{2 \over 3} \, , \qquad \lim _{r \to \infty} {\log \E \tau_r
\over \log r} = 3 \, , \qquad \lim _{n \to \infty} {\log |W_n|
\over \log n} = {2 \over 3} \quad a.s. \, ,$$ where $\tau_r$ is the hitting
time of distance $r$ from the origin (the expectation $\E$ is only
over the randomness of the walk) and $W_n$ is the range of the
random walk after $n$ steps.
\end{theorem}

Our main contribution is the analysis of the geometry of the IIC.
The IIC admits {\em fractal} geometry which is dramatically
different from the one of the infinite component of {\em
supercritical} percolation. The latter behaves in many ways as
$\Z^d$ after a ``renormalization'' i.e.~ignoring the local
structure \cite{GM} (see also \cite{G} for a comprehensive
exposition). In particular, the random walk on the supercritical
infinite cluster has an invariance principle, the spectral
dimension is $d_s=d$ and other $\Z^d$-like properties hold, see
\cite{DFGW, BM, B, SS, BB, MP}.

Our analysis establishes that balls of radius $r$ in the IIC
typically have volume of order $r^2$ and that the {\em effective
resistance} between the center of the ball and its boundary is of
order $r$. These facts alone suffice to control the behavior of
the random walk and yield Theorem \ref{mainthm}, as shown by
Barlow, J\'{a}rai, Kumagai and Slade \cite{BJKS}. The key
ingredient of our proofs is establishing that the critical
exponents dealing with the {\em intrinsic} metric (i.e., the
metric of the percolated graph) attain their mean-field values. It
was demonstrated first in the work of AN and Peres
\cite{NP3} that these exponents yield analogous statements to the
Alexander-Orbach conjecture in the {\em finite} graph setting. In
particular, in \cite{NP3}, the diameter and mixing time of
critical clusters in mean-field percolation on finite graphs were
analyzed.


In different settings the Alexander-Orbach conjecture was proved
by various authors. When the underlying graph is an infinite
regular tree, this was proved by Kesten \cite{K2} and Barlow and
Kumagai \cite{BK} and in the setting of {\em oriented} spread-out
percolation with $d>6$, this was proved recently in the
aforementioned paper \cite{BJKS}. \\

\subsection{Anomalous diffusion}
The fact that $d_s=4/3$ should best be contrasted against another
natural definition of dimension, and that is the volume growth
exponent $d_f$ defined, for any infinite connected graph $G$ by,
$$ d_f = \lim _{r \to \infty} { \log |B_G(x,r)| \over \log r } \qquad\qquad
\text{(if the limit exists)},$$ where $B_G(x,r)$ is the ball, in
the shortest-path metric with center $x$ and radius $r$, and
$|B_G(x,r)|$ is its volume, i.e.~the number of vertices of the
graph in it. The volume growth exponent is the graph analog of the
Hausdorff dimension. For the IIC this limit exists and is equal to
$2$ for all $x$ (Theorems \ref{exp} and \ref{lowerexp} below).
Hence we have two natural notions of dimension which give
different answers. For comparison, for $\Z^d$ we have $d_f=d_s=d$.
More generally, for any Cayley graph $d_f=d_s$ (indeed, Gromov's
celebrated result \cite{G81} shows that $d_f$ exists and is
integer for any Cayley graph, and then Theorem $5.1$ in \cite{HS}
shows that $d_f=d_s$) and there are other rich families which
satisfy this. To understand the discrepancy, we need to understand
{\em anomalous diffusion}.

Anomalous diffusion is the phenomenon that for many natural {\em
fractals}, or more precisely, graphical analogs of fractals,
random walk on the fractal is significantly slower than in
Euclidean space. In particular, while we expect a random walker on
$\Z^d$ to be at distance $t^{1/2}$ at time $t$, on a fractal we
find it in distance $t^{1/\beta}$ where $\beta\geq 2$ and often
the inequality is strict\footnote{$\beta$ is sometimes called the
``walk dimension'' and denoted by $d_w$.}. In fact, we now know
that any value of $\beta$ between $2$ and $d_f+1$ may appear
\cite{B2}. This phenomenon was first observed by physicists in the
context of disordered media \cite{AO, RT} i.e.~in our context.
Correspondingly, the first mathematical results are Kesten's
\cite{K2} who analyzed random walk on the IIC in $\Z^2$ and on an
infinite regular tree. For the IIC on a tree Kesten's results are
complete and he shows that $\beta=3$. On the IIC in two
dimensions, Kesten showed that the expected distance of the random
walk from the origin after $t$ steps is at most $t^{1/2 - \eps}$
for some $\eps>0$, hence $\beta>2$ if it exists. Despite the great
progress seen since on critical two-dimensional percolation, the
exact value of $\beta$ in this case is still unknown.

\begin{figure}
\includegraphics[scale=0.65]{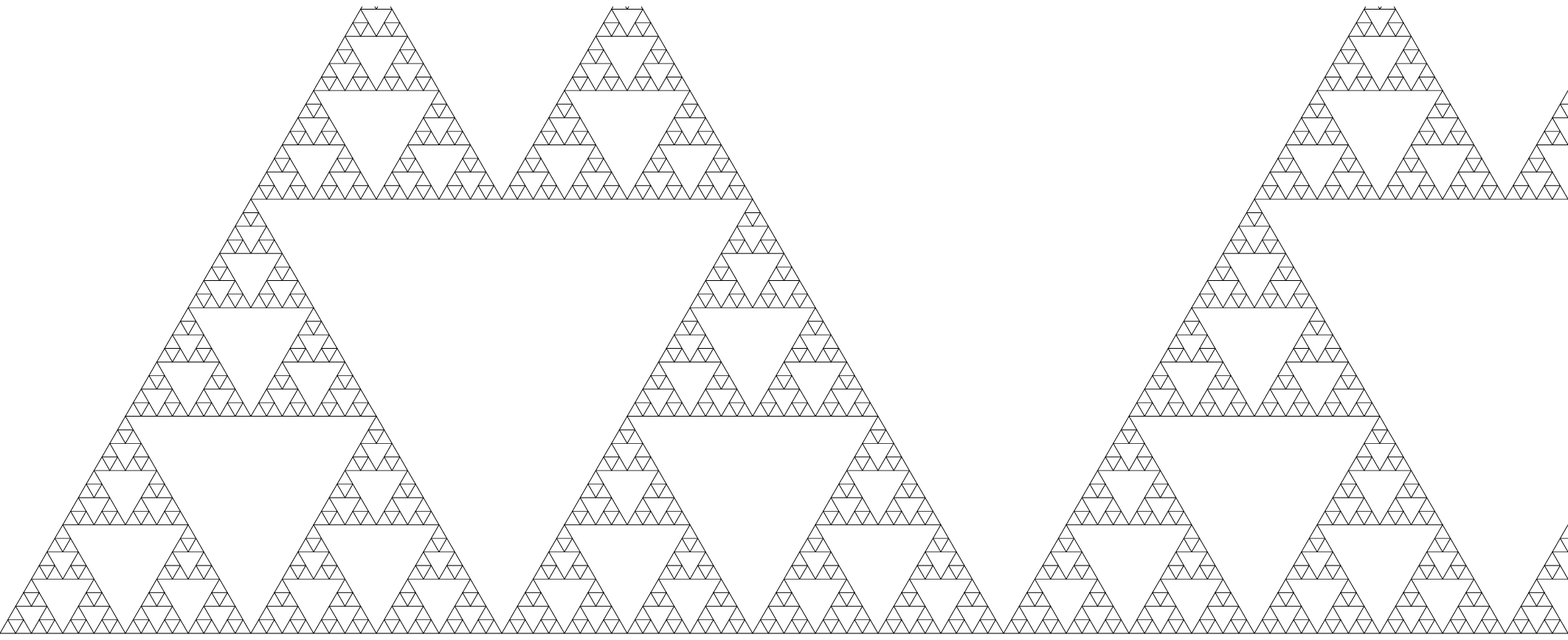}
\caption{\label{cap:gasket}A portion of the graphical Sierpinski
gasket.}
\end{figure}

The attention of the mathematical community then shifted to
regular fractals. The first to be analyzed were {\em finitely
ramified} fractals, namely fractals that can be disconnected by
the removal of a constant number of points in any scale. For
example, arbitrarily large portions of the Sierpinski gasket
(figure \ref{cap:gasket}) can be disconnected by the removal of
$3$ points. In these cases $\beta$ can be calculated explicitly,
for example for the Sierpinski gasket $\beta=\log 5/\log 2$
\cite{BP}. A significant step forward was done by Barlow and Bass
\cite{BBs, BBs2} who showed that on any generalized Sierpinski
carpet $\beta$ is well defined, and the random walk exhibits many
regularity properties analogous to those of random walk on $\Z^d$,
mutatis mutandis. Unfortunately, these techniques do not allow to
calculate $\beta$ for many natural examples, and this remains a
significant open problem.

For sufficiently ``well-behaved'' $G$ we expect that $\beta =
2d_f/d_s$. Heuristically this is easy to understand since if the
random walk reaches a distance of $\approx r=t^{1/\beta}$ it
should see $\approx r^{d_f}=t^{d_f/\beta}$ points and assuming
homogeneity we should have $\p_t(0,0)\approx t^{-d_f/\beta}$. This
explains the connection between the various ``exponents'' in
Theorem \ref{mainthm}. Among the results of Barlow and Bass
\cite{BBs,BBs2} is the proof that indeed $\beta = 2d_f/d_s$ for
any generalized Sierpinski carpet.

\subsection{Percolation}\label{perc} Bond percolation on a graph $G$ with parameter $p
\in [0,1]$ is a probability measure $\prob_p$ on random subgraph
of the $G$ obtained by retaining each edge independently with
probability $p$ and deleting it otherwise. Edges retained are
called {\em open} and edges deleted are called {\em closed}. The
graphs that interest us most are lattices in $\R^d$, in particular
$\Z^d$, and regular trees. It is well known that this model
exhibits a phase transition, that is, there exists a critical
probability $p_c(\Z^d)\in [0,1]$ such that for all $p>p_c$ almost
surely there exists an infinite connected cluster, and for any $p
< p_c$ almost surely all clusters are finite. Percolation at $p_c$
is called {\em critical} percolation.
The subcritical and
supercritical cases are understood quite well, and in neither case
is it reasonable to call the resulting graphs ``fractals''. The
subcritical case consists of finite clusters with exponential tail
on their size \cite{M, AB}. The supercritical case, as mentioned before, behaves in many ways as a perturbed version of $\Z^d$ and most interesting
quantities behave the same as on $\Z^d$. For instance, $d_f=d$, $\beta=2$,
and $d_s=d$. The structure of the resulting graph in critical percolation, however, is dramatically different.

It is widely believed that critical percolation does not exhibit
an infinite cluster almost surely. This has been established only
for the case $d=2$ by Kesten \cite{K0} and for sufficiently large
$d$ by Hara and Slade \cite{HaS0}. Proving it for all $d$ is
considered one of the most challenging problems in probability
theory. Nevertheless, for all $d>1$ it is known that in any scale there are
clusters comparable to the scale \cite[Theorem 1]{A}, and it is
conjectured that they have fractal-like properties.
Hence the natural question arises: what is the
corresponding spectral dimension $d_s$?
As already mentioned, there is a significant difference between low and high
dimensions. Let us therefore spend a little effort on the
difference between ``low'' and ``high'' dimensions in percolation.

Many models in mathematical physics exhibit an {\em upper critical
dimension} and for percolation this happens at $d=6$. The picture,
as developed by physicists, is that for $d>6$ the space is so vast
that different pieces of the critical cluster no longer interact.
The effect of this is that the geometry ``trivializes'' and for
most questions the answer would be as for percolation on an
infinite regular tree. This is also known as {\em mean-field}
behavior.
%
Aspects of this picture were confirmed rigorously but with one
important caveat. The technique used, {\em lace expansion}, is
perturbative and hence requires one of the following to hold:
\begin{itemize}
\item The dimension $d$ should be large enough ($d\geq 19$ seems to be the
  limit of current techniques).
\item The dimension should satisfy $d>6$ but the lattice needs to be
  sufficiently spread out. For example, one may take some $L$ sufficiently
  large and put an edge between every $x,y\in \Z^d$ with $|x-y|\leq L$.
\end{itemize}
Credit for these remarkable results goes to Hara and Slade
\cite{HaS0}. For $d<6$ it has been proved that percolation cannot
attain mean-field behavior \cite{CC}. Specifically, Hara and Slade
proved the that for these lattices the {\em triangle condition}
holds. The triangle condition, suggested as an indicator of mean-field
behavior by Aizenman and Newman \cite{AN} is
\be\label{triangle}
\sum_{x,y\in \Z^d} \prob_{p_c}(0\lr x)\prob_{p_c}(x\lr
y)\prob_{p_c}(0\lr y) <\infty
\ee
where $x\lr y$ denotes the event that $x$ is connected to
$y$ by an open path (for simplicity we assume that the set of
vertices of the lattice is always $\Z^d$ and denote the set of
edges by $E(\Z^d)$). To see how to analyze the behavior of critical
and near-critical percolation using the triangle condition, see
\cite{AN, BA, Ng}

A slightly different approach to mean-field behavior is via the {\em
two-point function} i.e.~the probability that $x$ is connected to
$y$ by an open path. It has the estimate, for
all $x,y\in \Z^d$,
\be \label{tpt}
\prob_{p_c} \big ( x \lr y \big
) \approx |x-y|^{2-d} \, ,
\ee where $\approx$ means that the
ratio of the quantities on the left and on the right is bounded by
two constants depending only on $d$ and $L$. Here and below we
abuse notation by considering that $0^{2-d}=1$. A simple calculation
shows that, when $d>6$, (\ref{tpt}) implies (\ref{triangle}) hence
the assumption on the two-point function is stronger. It was obtained
using the lace expansion by Hara, van der Hofstad and Slade
\cite{HaHS} for the spread-out model and $d>6$, and by Hara
\cite{Ha} for the nearest-neighbor model with $d\geq 19$ (in fact,
they obtained the right asymptotic behavior of (\ref{tpt}),
including the constant).

At present there is no known lattice in $\R^d$ for which the triangle condition is known and the two-point function is unknown (or
false). Nevertheless, We believe that there is value in noting which
results require the (formally) stronger two-point function estimate
(\ref{tpt}) and which require only the triangle condition. Reasons
to keep this distinction come from the fields of long-range
percolation \cite{BA, HHS07} and of percolation on
general transitive graphs \cite{Sc, Sc2}. In both cases the triangle
condition makes more sense and was proved in many interesting examples.
We will not dwell on these topics in this paper, but in general we
believe that any result we prove only using the triangle condition
should hold (perhaps with minor modifications) for long-range
percolation and for percolation on unimodular transitive graphs.

Returning to the Alexander-Orbach conjecture, our aim is to study
random walk on a typical large cluster. The term {\em incipient
  infinite cluster} was coined by Kesten, borrowing a vaguely-defined
term from the physics literature. His approach in \cite{K1} for the
two-dimensional case is to fix some integer $n$, to condition on the
event
$0 \lr \partial [-n,n]^2$ and then take $n \to \infty$.  In this paper
we take the approach suggested by van der Hofstad and J\'{a}rai \cite{HJ}, and
that is to fix some arbitrary far point $x$, condition on the
event $0 \lr x$, and then take $x\to\infty$. For both approaches
one still needs to show that the limit exists. This was done in
\cite{K1} for the case $d=2$, in \cite{HJ} for large $d$ as above
and in \cite{HHS} for the {\em oriented} percolation model with
$d>4$.


Formally, we endow the space of all configurations
$\{0,1\}^{E(\Z^d)}$ with the product topology (recall that $E(\Z^d)$ is the set
of edges of our lattice). We consider the conditional
measures given $0 \lr x$ and finally, the IIC is
the limit as $x\to \infty$ in the space of measures
$\mathcal{M}\big(\{0,1\}^{E(\Z^d)}\big)$ with the weak topology.
Put differently, for any cylinder
event $F$ (i.e., an event that can be determined by observing the
status of a finite number of edges) we have
\be \label{iic}
\prob_\mathrm{IIC}(F) = \lim _{|x| \to \infty} \prob_{p_c}(F \, \mid \,
0 \lr x) \, ,
\ee
\noindent where $p_c = p_c(\Z^d)$ is the
percolation critical probability. The convergence of the limit in
the right hand side, independently on how $x \to \infty$, is
proved in \cite{HJ} for $d$ large using the lace expansion. We
note in passing that the existence of the limit is not relevant
for our arguments. Indeed, even if the limit would not exist,
subsequence limits would exist due to compactness, and our results
would hold for each one. Thus the conclusions of Theorem \ref{mainthm}
hold for any lattice in $\R^d$ with $d>6$ for which the two-point
function estimate (\ref{tpt}) holds, and for
any IIC measure (i.e.~any subsequence limit as above).

\subsection{Intrinsic metric critical exponents}
The key ingredient in our proofs is showing that the {\em
intrinsic metric} critical exponents defined below assume their
mean-field values in high dimensions.

Let $G$ be a graph and write $G_p$ for the result of $p$-bond
percolation on it. Write $d_{G_p}(x,y)$ for the length of the
shortest path between $x$ and $y$ in $G_p$, or $\infty$ if there
is no such path. We call $d$ the {\em intrinsic metric} on $G_p$
--- other names in the literature include the {\em graph metric},
the {\em shortest-path
  metric} and even the {\em chemical distance}. From this point on, we always
perform critical
percolation with $p=p_c=p_c(\Z^d)$. Define the random sets
\begin{align*}
\bcr(x,r;G) &= \{ u : d_{G_{p_c}}(x,u) \leq r \} \, , \\
\partial \bcr(x,r;G) &= \{ u : d_{G_{p_c}}(x,u) = r \} \, .
\end{align*}
It will be occasionally important to take some $G \subset E(\Z^d)$
and sample $p_c(\Z^d)$-perco\-lation on $G$. Be careful not to
confuse the notation $\bcr(x,r;G)$ which refers to a random ball
in the percolation on $G$ with $B_G(x,r)$ which is just the
(deterministic) ball in $G$. In fact
$B(x,r;G)=B_{G_{p_c(\Z^d)}}(x,r)$.

We usually take $G$ to be $\Z^d$, and in
this case it would be suppressed from the notation. Our most
frequent notation is $B(x,r)$ which stands for $\bcr(x,r;\Z^d)$.

Define now the event
\begin{align*}
H(r;G) &= \Big \{ \partial B(0,r;G) \neq \emptyset \Big \} \, , \\
\intertext{and finally define} \Gamma(r)&=\sup_{G \subset E(\Z^d)}
\prob (H(r;G)) \, .
\end{align*}
Note again that we define $\Gamma$ by the maximum over all subgraphs of
$\Z^d$, but each one is ``tested'' with the $p_c$ of $\Z^d$ rather
than with its own $p_c$.

\begin{theorem}\label{exp} For any lattice $\Z^d$ with $d>6$
  satisfying the triangle condition (\ref{triangle}), there exists a
  constant $C>0$ such that

\renewcommand{\theenumi}{\roman{enumi}}
\renewcommand{\labelenumi}{(\roman{enumi})}

\begin{enumerate}

\item \label{enu:EB0r} \qquad $\E |\bcr(0,r)| \leq C r$

\item \label{enu:Gam<} \qquad $ \Gamma(r) \leq {C \over r} \, .$
\end{enumerate}
In particular, $\prob(H(r)) \leq C/r$.
\renewcommand{\theenumi}{(\arabic{enumi})}
\renewcommand{\labelenumi}{\arabic{enumi}}

\end{theorem}


The corresponding lower bounds to Theorem \ref{exp} are much
easier to prove and are not needed for the proof of Theorem
\ref{mainthm}. We state them for the sake of completeness.

\begin{theorem} \label{lowerexp} For any lattice $\Z^d$ with $d>6$
  satisfying the two-point function estimate
  (\ref{tpt}), there exists a constant $c>0$ such that

\begin{enumerate}

\item \qquad $\E |\bcr(0,r)| \geq c r \, ,$

\item \qquad $\prob(H(r)) \geq \frac{c}{r} \, .$
\end{enumerate}

\end{theorem}

The {\em extrinsic} metric corresponds to the shortest-path metric
in $\Z^d$ while the intrinsic metric corresponds to the (random)
shortest-path metric in the percolated graph $\Z^d_p$. The
classical {\em one-arm} critical exponent $\rho>0$ describes the
power law decay of the probability that the origin is connected to
sphere of radius $r$ {\em in the extrinsic metric}, that is
$$ \prob \Big ( \exists \, x \with |x|=r \hbox{ {\rm such that} } 0 \lr x) = r^{-1/\rho + o(1)} \, ,$$
where $|x|$ denotes the usual Euclidean norm. This exponent takes
the value $48/5$ in the two dimensional triangular grid, as shown
by Lawler, Schramm and Werner \cite{LSW} and Smirnov \cite{Sm}. In
the case of an infinite regular tree we have $\rho=1$ by a Theorem
of Kolmogorov \cite{Ko} (here the critical probability is $p_c={1
\over \ell-1}$, where $\ell$ is the vertex degree of the tree). In
high dimensions it was conjectured that $\rho=1/2$ (see \cite{S} and the upcoming paper \cite{KN} for a proof) ---
a surprising belief at first, since we expect critical exponents
in high dimensions to take the same value they do on a tree.

Measuring distance with respect to the {\em intrinsic} metric
offers a simple explanation of this discrepancy. Indeed, as the
extrinsic and intrinsic metrics on the tree are the same, we have
that on a tree $\prob(H(r))\approx r^{-1}$ and by Theorem
\ref{exp} above we learn that this is the same order in the
high-dimension lattices. Similar results exist for critical
Erd\H{o}s-R\'enyi random graphs \cite{NP3}. \\

\subsection{About the proof}
From the point of view of analysis of fractals, the IIC is one of
the simplest cases to handle because of its tree-like structure.
Indeed, the main difficulty is the proof of Theorem \ref{exp}.
Once that is proved the proof proceeds roughly as follows. Write
$B_{{\rm IIC}}(0,r)$ and $\partial B_{{\rm IIC}}(0,r)$ for the
corresponding shortest-path metric balls in the IIC. Firstly,
since $\Gamma(r)\leq r^{-1}$ and $\E|\bcr(0,r)|\approx r$ we learn
that $|B_{{\rm IIC}} (0,r)| \approx r^2$. Secondly, the intrinsic
metric exponents show that there are $\geq cr$ ``approximately
pivotal'' edges --- $\lambda$-lanes in the language of \cite{NP3}
--- between $0$ and $\partial B(0,r)$ (Lemma \ref{effres}) and
therefore the {\em electric resistance} $R_{\rm eff}$ between $0$
and $\partial B_{{\rm IIC}}(0,r)$ is $\approx r$. We conclude that
$B_{{\rm IIC}}(0,r)$ is a graph on approximately $r^2$ vertices
with effective resistance between $0$ and $\partial B_{{\rm
IIC}}(0,r)$ of order $r$ --- the same structure a critical
branching process conditioned to survive to level $r$ has with
high probability.

Now, there are many ways to connect electric resistance and volume
estimates to hitting times, and in fact we simply quote a
perfectly-tailored-for-our-needs result from \cite{BJKS} which
concludes the proof of the theorem. However, let us briefly
describe a somewhat different but very natural approach. It starts with the
fact \cite{CRRST} that, in any finite graph $G$, and for any two
vertices $x$ and $y$,
$$
\Hit(x,y)+\Hit(y,x)=2R_{\rm eff}(x,y)\cdot|E(G)|
$$
where $\Hit(x,y)$ is the expected hitting time from $x$ to $y$ (or
in other words, the left hand side is the expected {\em commute
time} between $x$ and $y$). Since in our case $|E(G)|\approx r^2$
we get that the commute time is $\approx r^3$. Now, in general the
commute time only bounds the hitting time $\Hit(0,\partial B_{{\rm
IIC}}(0,r))$ from above, but in {\em strongly recurrent} graphs
this turns out to be sharp \cite{KM}. Thus, in time $r^3$ the
random walk has walked only in $B_{{\rm IIC}}(0,r)$ and it can be
shown that the end point is approximately uniformly distributed
(the walk has mixed in $B_{{\rm IIC}}(0,r)$ in that time). Since
$|B_{{\rm IIC}}(0,r)|\approx r^2$ we get that
$\p_{r^3}(0,0)=r^{-2}$, as required. The details of this approach
are described in the setting of finite graphs in \cite{NP3} and
can be adapted to this case as well.

A natural approach towards the proof of the volume growth exponent (part (i)
of Theorem \ref{exp}) is to show that $\E|\partial B(0,2r)|\geq c(\E|\partial
B(0,r)|)^2$ which would show that if $\E|\partial B(0,r)|$ is too large for
some $r$, it will start exploding, leading to a contradiction. We were not
able to pull this approach directly --- $\partial B(0,r)$ is hard to analyze
--- our substitute is to show that $\E|B(0,2r)|\geq (c/r)(\E|B(0,r)|)^2$.
This can be proved using relatively standard ``inverse BK inequalities'' and
the same argument then applies.

The proof of the
one-arm exponent (part (ii) of Theorem \ref{exp}) uses the precise
determination of the exponent $\delta$ by Barsky and Aizenman
\cite{BA}, which allows us to use a regeneration argument to show,
roughly, that $\Gamma(r) \leq r(\Gamma(r/4))^2 + C/r$ (the second
term comes from the results of \cite{BA}), from which the estimate
follows by induction. The lengths of the proofs of both pieces are
equivalent, which might hide the fact that the proof of the
one-arm exponent was much harder for us to obtain.

A final note is due about the use of $\Gamma(r)$. It would have
been more natural to discuss only $\prob(H(r))$ rather than
$\Gamma(r)$. However, we need to use a regeneration argument.
Basically we claim that, once you reached a certain level $r$,
each vertex $v\in \partial \bcr(0,r)$ has probability $\leq
\Gamma(s)$ to ``reach'' to $\partial \bcr(0,r+s)$. Heuristically,
one would assume that it would work even with $H(s)$, because the
part of the cluster you already ``explored'', $\bcr(0,r)$ only
makes it more difficult to reach the level $r+s$. The problem is
that $H(r)$ is not a {\em monotone} event. In general, if you have
a graph $G$ satisfying $\partial B_G(0,r)=\emptyset$ and you
remove an edge, it could increase the distance to some vertex $v$,
pushing it outside of $B_G(0,r)$, and restoring the event
$\partial B_G(0,r)\neq\emptyset$. Hence it is not possible to use
the regeneration argument with $H(r)$ --- there is simply no
inequality in either direction relating $\prob(H(r))$ with the
conditional probability of $H(r)$ given some partial configuration
of edges. The use of $\Gamma(r)$ helps us circumvent this problem.
See the proofs of lemma \ref{effres} (page \pageref{effres}) and
of part (\ref{enu:Gam<}) of Theorem \ref{exp} (page
\pageref{page:II}).

\subsection{Organization and notation.}
In \S\ref{sec:AO} we show how the intrinsic metric critical exponents
(Theorem \ref{exp}), together with (\ref{tpt}) 
yield our main result, Theorem \ref{mainthm}. In \S\ref{sec:exp}
we derive the mean-field estimates of Theorem \ref{exp} and \ref{lowerexp}.


For $x,y \in \Z^d$ we write $x \lr y$ for the event that $x$ is
connected to $y$ by an open path. We write $x \lrr y$ if there is
an open path of length $\leq r$ connecting $x$ and $y$. In
order to improve readability, we denote constants which
depend only on $d$ and the lattice by $C$ (to denote a large constant) and $c$
(to denote a small constant) and as we do not attempt to optimize these
constants we frequently use the same notation to indicate
different constants. For two monotone events of percolation $A$
and $B$ we write $A \circ B$ for the event that $A$ and $B$ occurs
in disjoint edges and we often use the van den Berg and Kesten inequality (BK
for short) $\prob(A\circ B) \leq \prob(A)\prob(B)$ (see \cite{vdBK, G} or
\cite{BCR} for more details).




\vspace{.1 in}

\section{{\bf Deriving the Alexander-Orbach Conjecture from Theorem
    \ref{exp}}}\label{sec:AO}

In this entire section we assume the two-point function estimate (\ref{tpt}) and Theorem \ref{exp}.
We will use results of Barlow, J\'{a}rai, Kumagai and Slade \cite{BJKS} which
are stated for {\em random graphs} and hence are perfectly suited for our
case. It is interesting to note that $\log \log$ fluctuations really do exist,
and hence any result for {\em fixed graphs} will naturally be somewhat
imprecise.
To state the results of \cite{BJKS}, we need the following
definitions. Given an instance of the IIC (that is, an infinite
connected graph containing the origin) write $B_{{\rm IIC}}(0,r)$
and $\partial B_{{\rm IIC}}(0,r)$ for the ball of radius $r$
around $0$ and the boundary of the ball, respectively, in the
shortest path metric on the IIC. Denote by $R_{{\rm eff}}
(0,\partial B_{{\rm IIC}}(0,r))$ the effective resistance between
$0$ and $\partial B_{{\rm IIC}}(0,r)$ when one considers $B_{{\rm
IIC}}(0,r)$ as an electric network and gives each edge a
resistance of $1$
--- see \cite{DS} for a formal definition. For $\lambda >1$ we
write $J(\lambda)$ for the set of $r$'s for which the following
conditions hold:
\begin{enumerate}
\item $\lambda ^{-1} r^2 \leq  |B_{{\rm IIC}}(0,r)| \leq \lambda
r^2 \, ,$

\item $ R_{{\rm eff}} (0, \partial B_{{\rm IIC}}(0,r)) \geq
\lambda^{-1} r \, .$
\end{enumerate}
Theorems 1.5 and 1.6 of \cite{BJKS} relate the information of
volume and effective resistance growth to the behavior of random
walks. They can be stated as follows.

\begin{theorem}[\cite{BJKS}] \label{RW}
If there exist some constants $K, q>0$ such that for any large
enough $r$ we have \be \label{maincond} \prob_\mathrm{IIC} \Big (r
\in J(\lambda) \Big ) \geq 1 - K\lambda^{-q} \, , \ee then the
conclusions of Theorem \ref{mainthm}
hold.

\end{theorem}

We begin with some lemmas leading to the fact that condition
(\ref{maincond}) holds in our setting. We start with some volume
estimates.

\begin{lemma}\label{volupper} For any lattice $\Z^d$ with $d>6$ satisfying
(\ref{tpt}), there exists a
constant $C>0$ such that for any integer $r\geq 1$ and any $x\in
G$ with $|x|$ sufficiently large we have
$$ \E \big [ |\bcr(0,r)|\cdot {\bf 1} _{\{ 0 \lr x \} } \big ]\leq Cr^2
|x|^{2-d} \, .$$
\end{lemma}
Here and below ``$|x|$ sufficiently large'' means essentially that $|x|>4Lr$
where $L$ is the length of the longest edge in $E(\Z^d)$. This point will not
play any role, though.
\begin{proof}[{\bf Proof}] We have
$$ \E \big [ |\bcr(0,r)| \cdot {\bf 1} _{\{ 0 \lr x \} } \big ] = \sum _{z
\in \Z^d} \prob \big ( 0 \lrr z \, , \, 0 \lr x )\, .$$ If $0 \lrr
z$ and $0 \lr x$, then there must exist some $y$ such that the
events $0 \lrr y$, $y \lrr z$ and $y \lr x$ occur disjointly. So
\begin{align*}
\E \big [ |\bcr(0,r)| \cdot {\bf 1}_{\{ 0 \lr x \} } \big ] &=
\sum_z \prob(0\lrr z \, , \, 0\lr x) \\
&\leq \sum_{z,y}\prob(\{0\lrr y\}\circ \{y\lrr z\}\circ \{y\lr
x\}) \intertext{and applying the BK inequality twice,} &\leq
\sum_{z,y}\prob(0\lrr y)\prob(y\lrr z)\prob(y\lr x)
\intertext{which by (\ref{tpt}) and the fact that $|x-y|\geq
|x|/2$ when $x$ is sufficiently large}
&\leq C|x|^{2-d}\sum_{z,y}\prob(0\lrr y)\prob(y \lrr z)
\intertext{We now use part (\ref{enu:EB0r}) of Theorem \ref{exp}
to sum, first over $z$ and then over $y$. We get} &\leq
Cr|x|^{2-d}\sum_y \prob(0\lrr y) \leq Cr^2|x|^{2-d}.\qedhere
\end{align*}
\end{proof}


\begin{lemma} \label{vollower} For any lattice $\Z^d$ with $d>6$ satisfying
(\ref{tpt}), there exists a
constant $C>0$ such that for any $r\geq 1$, any $\eps < 1$ and any
$x\in \Z^d$ with $|x|$ sufficiently large we have that
$$ \prob \Big ( |\bcr(0,r)| \leq \eps r^2 \, , \, 0 \lr x
\Big ) \leq C \eps |x|^{2-d} \, .$$
\end{lemma}

\begin{proof}
If $|\bcr(0,r)| \leq \eps r^2$ then there must exists a (random)
level $j\in [r/2,r]$ in which $|\partial \bcr(0,j)| \leq 2\eps r$.
Fix the smallest such $j$. Now, if $0\lr x$ then there must be
some vertex $y\in\partial \bcr(0,j)$ which is connected to $x$
``off $\bcr(0,j-1)$'' i.e.~with a path that does not use any of
the vertices in $\bcr(0,j-1)$. Let therefore $A$ be some subgraph
of $\Z^d$ which is ``admissible'' for being $\bcr(0,j)$
i.e.~$\prob(\bcr(0,j)=A)>0$. We get
$$
\prob(0\lr x\;|\;\bcr(0,j)=A)\leq \sum_{y\in \partial A}
  \prob(y\lr x \mbox{ off }A\setminus \partial A\;|\;\bcr(0,j)=A)
$$
where $\partial A$ stands for the vertices in the graph $A$
furthest from $0$. A moment's reflection shows that, for any $A$ and any
$y\in\partial A$, the event
$\{y\lr x$ off $A\setminus \partial A\}$ is {\em independent} of
the event $\{\bcr(0,j)=A\}$. Therefore
we can write
$$
\prob(0\lr x\;|\;\bcr(0,j)=A)\leq \sum_{y\in \partial A}
  \prob(y\lr x \mbox{ off }A\setminus \partial A)
\leq\sum_{y\in \partial A}\prob(y \lr x)\leq C|\partial A||x|^{2-d}
$$
where the last inequality uses the two-point function estimate
(\ref{tpt}) and the fact that $|x-y|\geq |x|/2$. By the definition
of $j$ we have $|\partial A|\leq 2\eps r$ and summing over all
admissible $A$ gives
$$
\prob(|\bcr(0,r)|\leq\eps r^2, 0\lr x)\leq  C\eps
r|x|^{2-d} \cdot \sum_{A} \prob (\bcr(0,j)=A).
$$
However, the events $\bcr(0,j)=A_1$ and
$\bcr(0,j)=A_2$ are disjoint and the union of these over
all $A$ imply that $\partial\bcr(0,\frac{1}{2}r)\neq\emptyset$. Part
(\ref{enu:Gam<}) of Theorem \ref{exp} shows that the probability of this union
is $\leq C/r$, finishing our lemma.
\end{proof}

We continue with some effective resistance estimates. Recall the
following definitions from \cite{NP3}.

\begin{itemize}
\item An edge $e$  between $\partial \bcr(0,j-1)$ and
$\partial \bcr(0,j)$ is called a {\bf lane} for $r$ if it
there is a path with initial edge $e$ from
  $\partial \bcr(0,j-1)$ to  $\partial \bcr(0,r)$
that does not return to $\partial \bcr(0,j-1)$. \item Say that a
level $j$ (with $0<j<r$) has $\lambda$ {\bf lanes} for $r$ if
there are at least $\lambda$ edges between $\partial \bcr(0,j-1)$
and $\partial \bcr(0,j)$ which are lanes for $r$.

\item We say that $0$ is $\lambda$-{\bf lane rich} for $r$, if
more than half of the levels $j\in [r/4,r/2]$ have $\lambda$
lanes for $r$.
\end{itemize}

\noindent Recall also the Nash-Williams \cite{NW} inequality (see
also \cite[Corollary 9.2]{P}).
\begin{lemma} [\cite{NW}] \label{NW} If $\{\Pi_j\}_{j=1}^J$ are
disjoint cut-sets separating $v$ from $U$ in a graph with unit
conductance for each edge, then the effective resistance from $v$
to $U$ satisfies
$$ R_{{\rm eff}} (v \lr U) \geq \sum_{j=1}^J {1 \over |\Pi_j|} \, .$$
\end{lemma}

\begin{lemma}\label{eventrx} For any lattice $\Z^d$ with $d>6$ satisfying
(\ref{tpt}), there exists a constant $C>0$ such that for any $r\geq 1$, for any
event $E$ measurable with respect to $\bcr(0,r)$ and for
any $x\in\Z^d$ with $|x|$ sufficiently large,
$$
\prob(E\cap \{0\lr x\})\leq C\sqrt{r\prob(E)}\,|x|^{2-d}.
$$
\end{lemma}
\begin{proof} We first note that by Lemma \ref{volupper} there exists
some $j \in [r,2r]$ such that \be\label{boundedlevel} \E \big [
|\partial \bcr(0,j)|\cdot {\bf 1} _{\{ 0 \lr x \} } \big ] \leq C
r |x|^{2-d} \, . \ee Now fix some $M>0$ (which we shall optimize
in the end) and write
\begin{multline*}
\prob(E\cap \{0\lr x\}) \leq \prob(|\partial \bcr(0,j)|>M,\,0\lr
x) + \prob(E\cap \{|\partial \bcr(0,j)|\leq M, 0\lr x\}).
\end{multline*}
Now, for the first term we use (\ref{boundedlevel}) and Markov's
inequality and get
$$
\prob(|\partial \bcr(0,j)|>M,\,0\lr x) \leq
\frac{Cr|x|^{2-d}}{M}\,.
$$
For the second term we do as in Lemma \ref{vollower}. We condition
over $\bcr(0,j)$ and note that for any $A$ we have
$$
\prob(0\lr x\;|\;\bcr(0,j)=A)\leq C|\partial A||x|^{2-d}\leq
CM|x|^{2-d}\,.
$$
Summing over all subgraphs $A$ which satisfy $E$ (here is where we
use that $E$ is measurable with respect to $\bcr(0,r)$)
gives that the second term is $\leq\prob(E)\cdot CM|x|^{2-d}$.
Summing both terms we get
$$
\prob(E\cap \{0\lr x\}) \leq \frac{Cr|x|^{2-d}}{M}+C\prob(E)M|x|^{2-d}\,.
$$
Taking $M=\sqrt{r/\prob(E)}$ proves the lemma.\end{proof}

\begin{lemma} \label{effres} For any lattice $\Z^d$ with $d>6$ satisfying
(\ref{tpt}), there exists a constant $C>0$ such for any $r\geq 1$, any
$\lambda>1$ and any $x\in \Z^d$ with $|x|$ sufficiently large we have that
$$ \prob \Big ( R_{{\rm eff}} (0, \partial \bcr(0,r)) \leq \lambda^{-1} r \, , \, 0 \lr x
\Big ) \leq C \lambda^{-1/2} |x|^{2-d} \, .$$
\end{lemma}
\begin{proof} Let $j \in [r/4,r/2]$ and denote by $L_j$ the number of
lanes between $\partial \bcr(0,j-1)$ and $\partial\bcr(0,j)$. Let
us condition on $\bcr(0,j)$ and take some edge between $\partial
\bcr(0,j-1)$ and $\partial \bcr(0,j)$. Denote the end vertex of
this edge in $\partial \bcr(0,j)$ by $v$. The event that the edge
is a lane for $r$ implies that we have $\partial B(v,r/2;
G)\ne\emptyset$ in the graph $G$ that one gets by removing all
edges which are needed to calculate $\bcr(0,j)$, that is, all the
edges with at least one vertex in $\bcr(0,j-1)$. Thus, conditioned
on $\bcr(0,j)$, the event we are interested in is
$$
\partial B(v,r/2; G)\neq\emptyset \, ,
$$
By the definition of $\Gamma$ (with the $G$ from the definition of
$\Gamma$ being our $G$) and translation invariance we get that
$$
\prob(\partial B(v,r/2; G) \neq\emptyset \;\mid
\;\bcr(0,j))\leq\Gamma(r/2)
$$
which by part (\ref{enu:Gam<}) of Theorem \ref{exp} is $\leq C/r$. In total we
get
%
$$
\E [ L_j \mid \bcr(0,j)] \leq {C \over r} |\partial \bcr(0,j)|
\, .
$$  This together with part (\ref{enu:EB0r}) of Theorem
\ref{exp} implies that the expected number of lanes in $\bcr(0,r/2)
\setminus \bcr(0,r/4)$ is at most a constant $C$. We learn
that the probability that $0$ is $\lambda$-lane rich is at most $C
\lambda ^{-1} r^{-1}$. Observe that Lemma \ref{NW} implies that if
$0$ is {\em not} $\lambda$-lane rich, then $R_{{\rm eff}} (0,
\partial \bcr(0,r)) \geq c \lambda^{-1} r$ and hence
$$ \prob \Big ( R_{{\rm eff}} (0, \partial \bcr(0,r)) \leq \lambda^{-1}
r\Big ) \leq C \lambda ^{-1} r^{-1} $$ (where if $\partial
\bcr(0,r)=\emptyset$ then we consider the resistance to be
$\infty$). An appeal to Lemma \ref{eventrx} and we are done.
\end{proof}



\noindent {\bf Proof of Theorem \ref{mainthm}.}
By Theorem \ref{RW} it suffices to show that (\ref{maincond})
holds. Indeed, fix $r\geq 1$ and $x \in \Z^d$ with $|x|$ sufficiently large.
Markov's inequality with Lemma \ref{volupper} and the lower bound
for the two-point function (\ref{tpt}) shows that
$$ \prob \big ( |\bcr(0,r)| \geq \lambda r^2 \, \mid \, 0 \lr x
\big ) \leq C\lambda ^{-1} \, .$$ Lemmas \ref{vollower} and
\ref{effres} show that
$$ \prob \big ( |\bcr(0,r)| \leq \lambda^{-1} r^2 \, \mid \, 0 \lr x
\big ) \leq C\lambda ^{-1} \, ,$$ and
$$ \prob \big (  R_{{\rm eff}} (0, \partial \bcr(0,r)) \leq \lambda^{-1} r \, \mid \, 0 \lr x
\big ) \leq C\lambda ^{-1/2} \, .$$ Taking the limit as $|x| \to
\infty$ shows that (\ref{maincond}) holds with constants $K=3C$ and $q=1/2$, concluding the proof. \qed

\section{{\bf  Intrinsic Metric Critical Exponents}}\label{sec:exp}

\noindent In this section we prove Theorem \ref{exp}. Our assumption
on the lattice is therefore that it satisfies the triangle condition
(\ref{triangle}). In effect we will be using a variation known as the
{\em open} triangle condition
\be\label{opentriangle}
\lim _{K \to \infty} \sup_{w : |w| \geq K} \prob(0\lr x)\prob(x\lr y)\prob(y\lr w) = 0 \, .
\ee
For $\Z^d$ the triangle condition implies the open triangle condition --- see
\cite[lemma 2.1]{BA}. Hence, from now on we will only use (\ref{opentriangle}).

\subsection{Intrinsic metric volume exponent}\label{sub:G} Here we
prove part (\ref{enu:EB0r}) of Theorem \ref{exp}. We use the notation
$$ G(r) = \E |\bcr(0,r)| \, .$$
The main part of the proof is the following Lemma.
\begin{lemma} \label{volrecur} Under the setting of Theorem \ref{exp}, there exists a constant
$c_1>0$ such that for all $r$
$$ G(2r) \geq {c_1 G(r)^2 \over r} \, .$$
\end{lemma}

\noindent Let us first see how to use Lemma \ref{volrecur}.

\noindent {\bf Proof of part (\ref{enu:EB0r}) of Theorem
\ref{exp}.} Assume without loss of generality that $c_1<1$ in
Lemma \ref{volrecur} and take $C_1 > \max\{(2/c_1), 2^d\}$. Assume
by contradiction that there exists $r_0$ such that $G(r_0) \geq
C_1 r_0$. Under this assumption we prove by induction that for any
integer $k \geq 0$ we have $G(2^k r_0) \geq C_1^{k+1}r_0$. The
case $k=0$ is our assumption, and Lemma \ref{volrecur} gives that
$$ G(2^{k+1} r_0) \geq {c_1 G (2^k r_0)^2 \over 2^k r_0} \geq C_1^{k+2} r_0 \,
,$$ where in the last inequality we used the induction hypothesis
and the fact that $C_1 > 2/c_1$. This completes our induction.

Now, since the number of vertices of distance $r$ from the origin
is at most $Cr^d$ for some constant $C$ which depends on $d$ and on the
lattice, but not on $r$, we get that for any integer $k \geq 0$
$$ C(2^{k}r_0)^d \geq G(2^k r_0) \geq C_1^{k+1}r_0 \, ,$$
and since $C_1 > 2^d$ we arrive at a contradiction (for some $k$ sufficiently
large) which proves the upper bound on $G(r)$. \qed

\noindent The next lemma is used in the proof of Lemma
\ref{volrecur}.

\begin{lemma} \label{revBK} There exists some constant $c>0$ such that
$$ \sum _{x,y \in \Z^d} \prob \big ( \{0 \lrr x\} \circ \{x \lrr y\}
\big ) \geq cG(r)^2 \, .$$
\end{lemma}
\noindent {\bf Proof.} By translation invariance of $\Z^d$ it
suffices to prove that
$$ \sum _{x,y \in \Z^d} \prob \big ( \{0 \lrr x\} \circ \{0 \lrr y\}
\big ) \geq cG(r)^2 \, .$$ The proof requires that we separate
slightly the starting points of the two paths.
Hence we shall
prove that there exists some $K>0$ such that if $u,v \in \Z^d$
satisfy $|u-v|\geq K$, then \be\label{revBKstp0} \sum
_{x,y\in\Z^d} \prob \big (  \{u \lrrc x\} \circ \{v \lrrc y\} \big
) \geq {1 \over 2}G(r-K)^2 \, .\ee Assuming this, we will then
take $u,v$ to be antipodal vertices on the sphere of radius $K$
(in the usual norm) of $\Z^d$ so that $|u-v|\geq K$. To see that
the assertion of the lemma follows from the above claim, notice
that if it occurs that $\{u \lrrc x\} \circ \{v \lrrc y\}$, then
by changing the status of edges from closed to open only in the
ball $B(0,K)$ we can make the configuration belong to the event
$\{0 \lrr x\} \circ \{0 \lrr y\}$. We deduce that there exists
some $c(K)>0$ such that
$$ \prob \Big ( \{0 \lrr x\} \circ \{0 \lrr y\} \Big ) \geq c(K)
\prob \Big ( \{u \lrrc x\} \circ \{v \lrrc y\} \Big ) \, ,$$ and
so by (\ref{revBKstp0})
$$ \sum _{x,y \in \Z^d} \prob \Big ( \{0 \lrr x\} \circ \{0 \lrr y\} \Big )
\geq {c(K) \over 2} G(r-K)^2 \, .$$ And it is clear that $G(r)
\leq CK^d G(r-K)$ and the assertion of the lemma follows, if
only $K$ can be chosen independently of $r$.

We proceed to prove (\ref{revBKstp0}). For any $u, v \in \Z^d$ and
an integer $\ell >0$ (later we put $\ell=r-K$) we have
$$ \prob \big (  \{u \lrrell x\} \circ \{v \lrrell y\} \big ) \geq \prob
\Big ( u \lrrell x \and v \lrrell y \and \C(u) \neq \C(v) \Big )
\, ,$$ where $\C(u)$ and $\C(v)$ denote the connected components
containing $u$ and $v$, respectively. By conditioning on $\C(u)$
we get that the right hand side equals
\begin{eqnarray*}
  \sum_{\substack{A\subset E(\Z^d)\\
                  u\lrrell x, u\nleftrightarrow v \inn A}}
\prob(\C(u)=A)\prob \Big (v \lrrell y \mid \C(u) =A \Big ) \, .
\end{eqnarray*}
Note that for $A$ such that $u\nleftrightarrow v$ in $A$ we have
$\prob \big ( \{v \lrrell y\} \mid \C(u) =A \big ) = \prob \big (
v \lrrell y \off A \big )$ where the event $\{v \lrrell y$ off
$A\}$ again means that there exists an open path of length at most
$\ell$ connecting $v$ to $y$ which avoids the vertices of $A$. At
this point we can remove the condition that $u\nleftrightarrow v$
in $A$ since in this case the event $\{v \lrrell y$ off $A\}$ is
empty. We get
\begin{equation}\label{condcu}
\prob \big (  \{u \lrrell x\} \circ \{v \lrrell y\} \big ) \geq
  \sum_{\substack{A\subset E(\Z^d)\\
                  u\lrrell x \inn A}}
\prob(\C(u)=A)\prob \big ( v \lrrell y \off A \big ) \, .
\end{equation}
Now, since
$$
\prob ( u \lrrell x) \prob (v \lrrell y) =
\sum_{\substack{A \subset E(\Z^d)\\
                u \lrrell x \inn A}}
\prob(\C(u)=A)  \prob (\{v \lrrell y\}) \, ,$$
we deduce by (\ref{condcu}) that \begin{multline}
\label{revBKstp1} \prob \big ( \{u
\lrrell x\} \circ \{v \lrrell y\} \big ) \geq \prob (u \lrrell x) \prob (v \lrrell y) \\
- \sum _{\substack{A \subset E(\Z^d)\\
                   u \lrrell x \inn A}}
\prob(\C(u)=A) \prob \big ( v \lrrell y \on A \big ) \,
,\end{multline} where the event $\{v \lrrell y \on A\}$ means that
there exists an open path between $v$ and $y$ of length at most
$\ell$ and any such path must have a vertex in $A$. For any $A
\subset \Z^d$ we have
$$ \prob \big ( v \lrrell y \on A \big ) \leq \sum _{z \in A}
\prob \big ( \{v \lr z\} \circ \{z \lrrell y\} \big ) \, .$$
Putting this into the second term of the right hand side of
(\ref{revBKstp1}) and changing the order of summation gives that
we can bound this term from above by
$$\sum_{z \in \Z^d} \prob\big(u \lrrell x \, , \, u \lr
z \big ) \prob \big ( \{v \lr z\} \circ \{z \lrrell y\} \big ) \,
.$$ If $u \lrrell x$ and $u \lr z$ then there exists $z'$ such
that the events $u \lr z'$ and $z' \lr z$ and $z' \lrrell x$ occur
disjointly. Using the BK inequality we bound this sum
above by
$$ \sum_{z, z'\in \Z^d} \prob(u \lr z') \prob(z' \lr z) \prob (z'
\lrrell x) \prob(v \lr z) \prob(z \lrrell y) \, .$$
We sum this over $x$ and $y$ and use (\ref{revBKstp1}) to get that
$$ \sum _{x,y\in\Z^d} \prob \big (  \{u
\lrrell x\} \circ \{v \lrrell y\} \big ) \geq G(\ell)^2 -
G(\ell)^2 \sum _{z, z'\in \Z^d} \prob(u \lr z') \prob(z' \lr z) \prob(z \lr v) \, .$$
Since $|u-v|\geq K$, by the open triangle condition (\ref{opentriangle}) we can take $K$ large enough (independently of $r$) so that
$$ \sum _{z, z'\in \Z^d} \prob(v \lr z) \prob(z \lr z') \prob(z' \lr u) \leq {1 \over 2} \, ,$$
which immediately yields (\ref{revBKstp0}) and concludes our proof.
\qed

\begin{figure}
\input{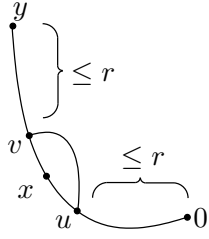}
\caption{\label{cap:overcounted}The couple $(x,y)$ is over-counted.}
\end{figure}
\medskip
\noindent {\bf Proof of Lemma \ref{volrecur}.} We start with a
definition. For an integer $K>0$ we say two vertices $x,y \in
\Z^d$ are {\em K-over-counted} if there exists $u,v \in \Z^d$ with
$|u-x|\geq K$ and $|v-x|\geq K$ such that
$$ \{0 \lrr u\} \circ \{v \lr x\} \circ \{x \lr u\} \circ \{u
\lr v\} \circ \{v \lrr y \} \, ,$$ see figure \ref{cap:overcounted}. Denote by
$N(K)$ the quantity
$$ N(K) = \Big | \Big \{ (x,y) \, : \, \{0 \lrr x\} \circ \{x \lrr
y\} \and (x,y) \hbox{ {\rm are not $K$-over-counted}} \Big \} \Big
| \, .$$
We claim that
\be \label{determ} N(K) \leq CK^{d} \cdot 2r
|\bcr(0,2r)| \ee Indeed, this deterministic claim follows by
observing that if $y \in \bcr(0,2r)$ and $\gamma$ is an open
simple path of length at most $2r$ connecting $0$ to $y$, then for
any $x \in \Z^d$ of distance at least $K$ from $\gamma$ (i.e., $x$
is of distance at least $K$ from {\em every} vertex of $\gamma$)
satisfying $\{0 \lrr x\}\circ \{x \lrr y\}$ the pair $(x,y)$ is
$K$-over-counted. To see this, let $\gamma_1$ and $\gamma_2$ be
disjoint open simple paths of length at most $r$ connecting $0$ to
$x$ and $x$ to $y$ respectively and take $u$ to be the last point
on $\gamma \cap \gamma_1$ and $v$ the first point on $\gamma \cap
\gamma_2$ where the ordering is induced by $\gamma_1$ and $\gamma_2$
respectively. Hence the map
$(x,y)\mapsto y$ from $N(K)$ into $B(0,2r)$ is at most $CK^d\cdot 2r$ to $1$,
which shows (\ref{determ}).

We now estimate $\E N(K)$. For any $(x,y)$ the BK inequality and
(\ref{tpt}) implies that the probability that $(x,y)$ are
$K$-over-counted is at most
$$ \sum_{\substack{u:|u-x|\geq K\\v: |v-x|\geq K}} \prob(0 \lrr u) \prob(v \lrr y) \prob(x \lr v) \prob(v \lr u) \prob(u \lr x) \, .$$
Writing $v'=v-x$ and $u'=u-x$ and using translation invariance we
get that this sum equals
$$
 \sum_{\substack{u': |u'|\geq K\\v': |v'|\geq K}}
  \prob(-x \lrr u') \prob(v' \lrr y-x) \prob (0 \lr v') \prob(v' \lr u') \prob (u' \lr 0) \, .
$$
We sum this over $y$ and then over $x$
and get that \begin{align*} \E \Big [ \big | \big \{ (x,y)
\hbox{ {\rm are $K$-over-counted}}\big \} \big | \Big ] &\leq
G(r)^2
\sum_{\substack{u':|u'|\geq K\\v':|v'|\geq K}} \prob (0 \lr v') \prob(v' \lr u') \prob (u' \lr 0) \, .
\end{align*}
This together with the triangle condition (\ref{triangle}) and Lemma \ref{revBK} gives that for some small $c>0$ we can choose some large $K$ such that $\E N(K) \geq c G(r)^2$. We take expectations in (\ref{determ}) and plug the
estimate $\E N(K) \geq c G(r)^2$ in to get the assertion of the
lemma. This concludes the proof of part (i) of Theorem \ref{exp}.
\qed





\newcommand{\history}
{
For the proof of proposition \ref{prop:GamH} we require the
following definitions. Fix some $K>1$. We say two paths $\beta,
\gamma$ are $(K,r)$-paths if $\beta, \gamma$ are two paths
starting at the origin and ending at the same vertex $v$ such that
$\len(\gamma)\geq Kr$ and $\len(\beta)\leq r$, and such that $\beta$ is a
shortest path between the origin and $v$. Given such paths,
an {\em excursion} of $\gamma$ is a sub-path of $\gamma$ such that
only the beginning and end vertex of it belong to $\beta$. It is
clear that $\gamma$ can be partitioned into disjoint excursions.
For two such paths $\beta, \gamma$ and an integer $s>0$ define
$$ E_s(\beta,\gamma) = \Big | \big \{ x \in \gamma \, : \, \hbox{{\rm the excursion containing $x$ has length}} \in [s,2s) \big \} \Big
|\, .$$ Proposition \ref{prop:GamH} follows from the four
lemmas below, bounding the rate of decay of $E_s(\beta,\gamma)$ in
four different regimes of $s$.

\begin{lemma}\label{endcase} We have
$$ \prob \Big ( \exists \,\, (K,r){\hbox{{\rm -paths} }} \beta,
\gamma \hbox{ {\rm such that $\gamma$ has an excursion of length
at least $r$}} \Big ) \leq Cr^{-3/2} \, .$$
\end{lemma}

\begin{lemma} \label{case1} There exists some large $K_1$ such that for any $s <
r^{1/d}$ we have
$$ \prob \Big ( \exists \,\, (K,r){\hbox{{\rm -paths} }} \beta, \gamma \with E_s(\beta,
\gamma) \geq K_1 r s^{3-d/2} \Big ) \leq e^{-\sqrt{r}} \, .$$
\end{lemma}

\begin{lemma} \label{case2} There exists some large $K_2$ such
that for any $s \in [r^{1/d}, r^{1/(d/2-2.1)})$ we have
$$ \prob \Big ( \exists \,\, (K,r){\hbox{{\rm -paths} }} \beta, \gamma \with E_s(\beta,
\gamma) \geq K_2 r s^{3.1-d/2} \Big ) \leq r^{-3} \, .$$
\end{lemma}

\noindent (the value $3.1$ can be replaced with any value $>3$)

\begin{lemma} \label{case3} There exists some large $K_3$ such
that for any $s \in [r^{1/(d/2-2.1)},r]$ we have
$$ \prob \Big ( \exists \,\, (K,r){\hbox{{\rm -paths} }} \beta, \gamma \with E_s(\beta,
\gamma) \geq K_3 s \Big ) \leq r^{-3} \, .$$
\end{lemma} \vspace{.2in}

\noindent Again we first show how to use these lemmas



\noindent{\bf Proof of proposition \ref{prop:GamH}.} We show
that for large enough $K$ we have $\Gamma(Kr) \leq 2 H(r)$ for all
$r$. To that aim we estimate $\prob ( \Gamma(Kr) \setminus H(r)
)$. If $\Gamma(Kr)$ occurs then there is some path $\gamma$ of
length at least $Kr$. If, in addition, $H(r)$ does not occur then
there exists some path $\beta$ of length at most $r$ from $0$ to
$\gamma(\len(\gamma))$ and we may assume $\beta$ has minimal length. If
$\gamma$ has no excursion of length at
least $r$ then $E_s(\beta, \gamma)=0$ for all $s > r$ and hence
$$ \sum _{n=1} ^{\lfloor \log_2 r \rfloor}
E_{2^n}(\beta,\gamma) = \len(\gamma) \geq Kr \, .$$ Summing up the
estimate of Lemmas \ref{case1}, \ref{case2} and \ref{case3} imply
that there is a $K$ such that \be\label{finish} \prob \Big (
\exists \,\, \beta, \gamma \with \sum _{n=1}^{\lfloor \log_2 r
\rfloor} E_{2^n}(\beta,\gamma) \geq Kr \Big ) \leq r^{-3} \log r
\, .\ee This together with Lemma \ref{endcase} yields that
$$ \prob (\Gamma(Kr)) \leq \prob (H(r)) + 2Cr^{-3/2} \, ,$$
which together with part (\ref{enu:H>}) of Theorem \ref{exp} concludes our
proof. \qed

\noindent We now proceed with the proofs of Lemmas \ref{endcase},
\ref{case1}, \ref{case2} and \ref{case3}.

\noindent {\bf Proof of Lemma \ref{endcase}.} Denote the required
event by $B_0$
$$B_0 = \Big \{ \exists \,\, (K,r){\hbox{{\rm -paths} }} \beta, \gamma
\hbox{ {\rm such that $\gamma$ has an excursion of length at least
$r$}} \Big \} \, .$$ If $B_0$ occurs then there exists $y\in \Z^d$
such that there is a simple open cycle starting and ending at $y$
of length at least $r$. We apply Lemma \ref{longpath} to get
$$ \prob (B_0) \leq C \sum _{y \in \Z^d} \prob \big ( 0 \lrr y \big ) r^{1-d/2} \leq C r^{2-d/2} \, ,$$
where the last inequality follows from part (\ref{enu:EB0r}) of Theorem
\ref{exp}. Since $d \geq 7$ we get the required result. \qed


Before the proof of Lemma \ref{case1}, we provide a small
technical claim that will assist us later on.

\begin{claim} \label{technical} Let $\QQ$ be the set of all
mappings $Q:\{1,\ldots, q\} \to \N$ satisfying
$$ \big | \{ j \, : \, Q(j) = i \} \big | \leq q2^{2-i} .$$
Then there exists some constant $C>0$ such that $|\QQ|\leq C^q$.
\end{claim}
\begin{proof}[{\bf Proof}] The number of possible options for
$Q^{-1}(i)$ is
$$ \sum _{k=0}^{q2^{2-i}} {q \choose k} \leq (C2^i)^{q2^{2-i}} \,
.$$ To see the last inequality note that when $i\geq 3$ the last
term in the sum is the dominant one and use the inequality ${q
\choose k} \leq (eq/k)^k$. Multiplying these up gives that the
number of possible $Q$'s is at most
\begin{equation*}
\exp \big ( C \sum_i q i 2^{2-i}  \big ) \leq \exp(Cq) \, .\qedhere
\end{equation*}
\end{proof}
\noindent {\bf Proof of Lemma \ref{case1}.} Denote by $B_1$ the
required event
$$ B_1 = \Big \{ \exists \,\, (K,r){\hbox{{\rm -paths} }} \beta, \gamma \with E_s(\beta,
\gamma) \geq K_1 r s^{3-d/2} \Big \} \, .$$ Assume that $B_1$
holds and take some $(K,r)$-paths $\beta, \gamma$.
Clearly the number of excursions of length in $[s,2s)$ of $\gamma$
is at least $E_s /2s$ and hence, if $B_1$ holds, then we have at
least ${1 \over 2} K_1 rs^{2-d/2}$ such excursions. Each excursion
$\delta$ starts at a vertex $v \in \beta$ and ends at some $w \in
\beta$ where we have $v$ preceding $w$ in the order induced by
$\beta$ (it is possible that in $\gamma$ the vertex $w$ is
traversed before $v$ but we ignore the order induced by $\gamma$).
It is possible that the same vertex $v$ appear as the starting
vertex of two different excursions. For example, if $v_1, v_2$ and
$v_3$ are vertices such that $v_i$ appears in $\beta$ before
$v_{i+1}$ but $\gamma$ has an excursion between $v_2$ and $v_1$
and then between $v_1$ and $v_3$. In that case $v_1$ will appear
twice as an excursion starting vertex. See figure \ref{cap:v1v2v3}.

\begin{figure}
\input{v1v2v3.pstex_t}
\caption{\label{cap:v1v2v3}The vertex $v_1$ is the starting vertex for two
  different excursions. Arrows indicate direction from $0$ to the common end
  point.}
\end{figure}

However, no vertex can appear more than twice in the role of an
excursion starting vertex. Also, if a vertex does appear twice in
that role, then the respective excursion ending vertices must be
different. All this means that we can order the excursions
$(\delta_1, \delta_2, \ldots)$ by the order of their respective
excursion starting vertices $v_i$. If two excursions have the same
starting vertex, we order them by the order of their ending
vertices. We set an integer $q$ \be\label{qval} q = \Big \lfloor
{1 \over 16} K_1 r s^{1-d/2} \Big \rfloor \, ,\ee and we say that
a $q$-tuple of excursions $(\delta_{i_1}, \ldots, \delta_{i_q})$
is {\em admissible} if $i_{j+1}-i_j \geq 4s$ for all $j$. As we
said before, if $B_1$ holds, then we have at least ${1 \over 2}
K_1 rs^{2-d/2}$ excursions and it is thus possible to find at
least $(4s)^q$ admissible tuples just by taking $i_j$ to be any
number in $[(8j-2)s, (8j+2)s]$.

Let $(i_1,\ldots, i_q)$ be an admissible tuple and let $v_j$ and
$w_j$ be the starting and ending vertices of $\delta_{i_j}$
respectively. Since $\beta$ is a shortest path, the number of
vertices in $\beta$ between $v_j$ and $w_j$ is no more than $2s$.
Hence, since $i_{j+1}-i_j \geq 4s$ it must be that $\beta$ visits
$w_j$ before $v_{j+1}$ (here we used the fact that the excursion
are ordered by their starting vertex). We learn that the event
$B_1$ implies that there are at least $(4s)^q$ tuples $(v_1,w_1,
v_2,w_2, \ldots, v_q, w_q)$ with the following properties:
\begin{enumerate}
\item Between every $v_j$ and $w_j$ there are two simple open
paths $\beta_j$ and $\delta_j$ such that $\len(\beta_j)\leq 2s$
and $\len(\delta_j) \in [s,2s)$.

\item Between every $w_{j-1}$ and $v_j$ there exists an open path
$\beta_j'$ which have accumulated length of at most $r$.

\item All the paths $\delta_j, \beta_j$ and $\beta_j'$ are edge
disjoint.
\end{enumerate}

Denote by $X_1(v_1,w_1, \ldots, v_q,w_q)$ the event that $(1)$
occurs and the paths $\delta_j$ and $\beta_j$ are disjoint for all
$j$, and by $X_2(v_1,w_1, \ldots, v_q,w_q)$ the event that $(2)$
occurs and all the paths $\beta_j'$ are disjoint. The BK inequality
and Markov's inequality gives that
\be
\label{lastmove} \prob(B_1)
\leq (4s)^{-q} \sum_{v_1,\ldots, w_q} \prob\big(X_1(v_1, \ldots,
w_q)\big) \prob\big(X_2(v_1, \ldots, w_q)\big) \, .
\ee
To bound $\prob(X_1)$ we use the BK inequality to get
$$
\prob\big(X_1(v_1, \ldots, w_q)\big) \leq \prod_{j=1}^q
\prob(v_j \stackrel{\, 2s}{\leftrightarrow} w_j)\prob(\exists
\delta_j \hbox{ {\rm from $v_j$ to $w_j$}} \with
\len(\delta_j)\geq s) \, ,
$$
and Lemma \ref{longpath} gives that
\be
\label{x1est} \prob\big(X_1(v_1, \ldots, w_q)\big) \leq \big(C
s^{1-d/2}\big)^q \prod_{j=1}^q \prob(v_j \stackrel{2s}
{\leftrightarrow} w_j) \, .
\ee
To estimate $\prob(X_2)$ we need
to examine the lengths of $\beta_j'$. Define a mapping $Q : \{1
,\ldots, q\} \to \N$ by
$$ Q(j)=\begin{cases}
1 & \len(\beta_{j}')<\frac{2r}{q}\\
i &
\len(\beta_{j}')\in\left[\frac{2^{i-1}r}{q},\frac{2^{i}r}{q}\right).\end{cases}$$
In particular $2^{Q(j)} \leq 2 q \len(\beta_j')/r + 2$. Since
$\sum_j \len(\beta_j') \leq r$ we have that \be\label{qcond} \sum
_{j=1}^q 2^{Q(j)} \leq 4q \, ,\ee and in particular
$$ \big | \{ j \, : \, Q(j)=i \} \big | \leq q 2^{2-i} \, .$$
Denote by $\QQ$ the set of all $Q$'s satisfying (\ref{qcond}). By the
BK inequality we have
$$\prob\big(X_2(v_1, \ldots, w_q)\big) \leq \sum_{Q \in \QQ} \prod_{j=1}^q \prob \Big( w_{j-1} \stackrel{\,\, {2^{Q(j)} r \over q}}{\longleftrightarrow} v_j \Big ) \, .$$
Putting this together with (\ref{x1est}) we get
$$ \prob(X_1)\prob(X_2) \leq \big(C s^{1-d/2}\big)^q \sum_{Q \in \QQ} \prod_{j=1}^q \prob(v_j
\stackrel{\, s}{\leftrightarrow} w_j) \prod_{j=1}^q \prob \Big(
w_{j-1} \stackrel{\,\, {2^{Q(j)} r \over q}}{\longleftrightarrow}
v_j \Big ) \, .$$ By translation invariance we can write this as
$$ \prob(X_1)\prob(X_2) \leq \big(C s^{1-d/2}\big)^q \sum_{Q \in \QQ} \prod_{j=1}^q
\prob(0 \stackrel{\, s}{\leftrightarrow} w_j-v_j) \prod_{j=1}^q
\prob \Big(0 \stackrel{\,\, {2^{Q(j)} r \over
q}}{\longleftrightarrow} v_j -  w_{j-1}\Big ) \, .
$$
We now sum this over all $v_j$ and $w_j$ and change the order of summation to
get
$$
\sum _{v_1, \ldots, w_q} \prob(X_1)\prob(X_2) \leq \big(C s^{1-d/2}\big)^q
\sum_{Q \in \QQ} \Big [ \Big (\sum_v \prob( 0 \stackrel{2s}{\leftrightarrow} v
) \Big)^q
\prod _{j=1}^q \sum_v  \prob \Big (0 \stackrel{\,\, {2^{Q(j)} r
\over q}}{\longleftrightarrow} v \Big ) \Big ] \, .
$$
Part (\ref{enu:EB0r}) of Theorem \ref{exp} gives that
$$ \sum_v \prob( 0 \stackrel{\,s}{\leftrightarrow} v ) \leq Cs \,
,$$ and hence
$$ \sum _{v_1, \ldots, w_q} \prob(X_1)\prob(X_2) \leq \big(C
s^{2-d/2}\big)^q \sum_{Q \in \QQ} \Big ( \prod _{j=1}^q C
{2^{Q(j)} r \over q} \Big ) \, .$$ The arithmetic-geometric mean
inequality gives that
$$ \prod _{j=1}^q C{2^{Q(j)} r \over q} \leq \Big ( {Cr \over q^2} \sum_j 2^{Q(j)} \Big
)^q \leq \Big ( {Cr \over q} \Big )^q \, ,$$ where the last
inequality is due to (\ref{qcond}). Applying Claim
\ref{technical} gives
$$ \sum _{v_1, \ldots, w_q} \prob(X_1)\prob(X_2) \leq \left ( {C
r s^{2-d/2} \over q}  \right ) ^q \, .$$ Putting this into
(\ref{lastmove}) and recalling the value of $q$ in (\ref{qval})
gives that
$$ \prob (B_1) \leq \Big ( {C \over K_1} \Big )^ q \, .$$
Since $s \leq r^{1/d}$ we have that $q \geq c\sqrt{r}$ for some
small $c>0$ and we conclude the proof by choosing $K_1$ large
enough.  \qed

\medskip
\noindent{\bf Proof of Lemma \ref{case2}.}
The argument is similar to that of the previous lemma, but we need a different
definition for admissible tuple. There we worked around the
fact that the various excursions intertwine over $\beta$ in complicated ways
(in other words, that the segments $(v,w)$ could intersect) by simply
selecting a collection of $q$ excursions that did not intertwine. This costed
a factor of $s$ in the number of moments $q$ but we didn't care. For larger
$s$ we can no longer pay this factor. On the other hand, we no longer need to
be careful with factors of the type $q^q$ and this simplifies some of the
calculations.

Denote by $B_2$ the
required event
$$ B_2 = \Big \{ \exists \,\, (K,r){\hbox{{\rm -paths} }} \beta, \gamma \with E_s(\beta,
\gamma) \geq K_2 r s^{3.1-d/2} \Big \} \, .$$
As in the previous lemma, assume the decomposition of $\gamma$ into excursions
$(\delta_1, \delta_2, \ldots)$ such that the excursions are
ordered by the order of their beginning vertex. For $i_1 < \cdots
< i_k$ write $v_j,w_j\in \beta$ for the respective starting and
ending vertices of excursion $\delta_{i_j}$ for $1 \leq j \leq k$.
Also write $\beta[v_j,w_j]$ for the sub-path of $\beta$ starting
at $v_j$ and ending at $w_j$. For an integer $k>0$ denote by
$Y(k)$ the event
$$ Y(k) = \Big \{ \exists \, i_1 < \cdots < i_k \with \beta[v_j,w_j] \cap \beta[v_1,w_1] \neq \emptyset \hbox{ {\rm for} } j=2,\ldots, k \Big \} \, .$$
Observe that the event $Y(k)$ is equivalent to saying that there
exists $i_1 < \cdots < i_k$ such that $v_j$ precedes $w_1$ in
$\beta$ for all $1 \leq j \leq k$. We first show that \be
\label{yk} \prob (Y(k)) \leq C_k rs^{k(3-d/2)} \, .\ee To see this
let $i_1 < \cdots < i_k$ be as in the definition of $Y(k)$ and let
$x_1, \ldots ,x_{2k}$ be the total collection of $v_j$ and $w_j$,
ordered by their appearance in $\beta$. If $Y(k)$ occurs we must
have the following:
\begin{enumerate}
\item There exists a path $\beta_0$ from $0$ to $x_1$ with
$\len(\beta_0) \leq r$.

\item There exists simple open paths $\beta_j$ between $x_j$ and
$x_{j+1}$ with $\len (\beta_j) \leq 2s$. This is because $\beta$
is a shortest path so the length of $\beta(v_1,w_1)$ at most $2s$
and all other $v_j$'s precede $w_1$ in $\beta$.

\item There exists a perfect matching of the $x_j$ i.e.~a function
$\varphi:\{1,\ldots,2k\}\to\{1,\ldots,2k\}$ with $\varphi(\varphi(j))=j$
such that between every couple $x_j, x_{\varphi(j)}$ there exists a simple
open path $\delta_j$ with $\len(\delta_j)\geq s$.

\item The paths $\beta_j$ and $\delta_j$ are edge-disjoint (note
that there are only $k$ different $\delta_j$).
\end{enumerate}
Write $Y(x_1, \ldots, x_{2k})$ for the above event. Since the
number of perfect matchings $\varphi$ is $2^k k!$, the BK inequality
and Lemma \ref{longpath} give that
$$ \prob (Y(x_1,\ldots, x_{2k})) \leq 2^k k! \big ( Cs^{1-d/2} \big)^k \prob ( 0
\lrr x_1 ) \prod_{j=1}^{2k-1} \prob ( x_j \stackrel{
2s}{\leftrightarrow} x_{j+1} ) \, .$$ By translation invariance we
have
$$ \prob (Y(x_1,\ldots, x_{2k})) \leq 2^k k! \big ( Cs^{1-d/2} \big)^k \prob ( 0
\lrr x_1 ) \prod_{j=1}^{2k-1} \prob ( 0 \stackrel{
2s}{\leftrightarrow} x_{j+1}-x_j ) \, .$$ We sum this over $x_1,
\ldots, x_{2k}$ and use part $(1)$ of Theorem \ref{exp} to get
$$ \prob (Y(k)) \leq 2^k k! \big ( Cs^{1-d/2} \big)^k \cdot Cr
\cdot (Cs)^{2k-1} \, ,$$ which gives (\ref{yk}).

Since $s \geq r^{1/d}$ and $d \geq 7$ we can fix some large $k$
such that
$$ \prob (Y(k)) \leq r^{-4} \, ,$$
and continue in a similar way to Lemma \ref{case1} only we take
$q$ this time to be a large constant (independent of $r$) to be
chosen later. We say a $q$-tuple $i_1 < \cdots < i_q$ is {\em
admissible} if $i_{j+1} - i_j > k$. As before, if $B_2$ occurs,
then there are at least $L={1 \over 2} K_2 r s^{2.1 - d/2}$
excursions with length in $[s,2s)$. We get that if
$$ \bigg | i_j - j{L \over q+1} \bigg | < {1 \over 2} \bigg ( {L
\over q+1} -k \bigg ) \, ,$$
then the tuple $(i_1, \ldots, i_q)$ is
admissible. Hence, if $B_2$ occurs then there are at least
$(L/(q+1)-k)^q$ admissible tuples. Observe that if $Y(k)$ does not
occur then for any admissible tuples $i_1 < \cdots < i_q$ the
vertex $w_{j}$ (the end vertex of $\delta_{i_j}$) precedes
$v_{i_{j+1}}$ (the beginning vertex of $\delta_{i_{j+1}}$) in the
order induced by $\beta$. Denote this event by $X(v_1, \ldots,
w_q)$ and so we can repeat the argument of Lemma \ref{case1}
\be\label{eq:B2X1X2}
\prob (B_2\setminus Y(k)) \leq \bigg(\frac{L}{q+1}-k\bigg)^{-q} \sum _{v_1, \ldots, w_q}
\prob (X_1(v_1, \ldots, w_q))\prob (X_2(v_1, \ldots, w_q)) \, ,
\ee
where $X_1$ and $X_2$ are defined as in the previous lemma. The BK
inequality yields
\begin{multline*}
 \prob(X_1)\prob(X_2) \leq \prod_{j=1}^q \prob(v_j \stackrel{\,
s}{\leftrightarrow} w_j) \cdot\\
\prob(\exists \delta_j \hbox{ {\rm from
$v_j$ to $w_j$ with} } \len(\delta_j)\geq s) \prod_{j=1}^q
\prob(w_{j-1} \lrr v_j) \, ,
\end{multline*}
and by Lemma \ref{longpath} we get
$$ \prob(X_1)\prob(X_2) \leq \big (C s^{1-d/2}\big)^q \prod_{j=1}^q \prob(v_j \stackrel{\,
s}{\leftrightarrow} w_j) \prod_{j=1}^q \prob(w_{j-1} \lrr v_j) \,
.$$ We sum over $v_1, \ldots , w_q$ and use translation invariance
and part $(1)$ of Theorem \ref{exp} as in the previous lemma to
get
$$
\sum_{v_1, \ldots ,w_q} \prob(X) \leq \big ( Cr s^{2-d/2} \big
)^q \, .
$$
Inserting this and the value of $L$ into (\ref{eq:B2X1X2}) gives
that
$$ \prob (B_2 \setminus Y(k)) \leq { ( Cr s^{2-d/2} \big
)^q \over \big ( {1 \over 2} K_2 r s^{2.1 - d/2}/(q+1) - k \big
)^q }\, .$$ We put $K_2 = 4(q+1)k$ and since $s \leq
r^{1/(d/2-2.1)}$ we get
$$ \prob (B_2 \setminus Y(k)) \leq \Big ( {C \over K_2 s^{0.1} }
\Big )^q \, ,$$ and since $s \geq r^{1/d}$ it is clear that we can
take $q$ large enough (but fixed) such that $\prob(B_2 \setminus
Y(k)) \leq r^{-4}$. The lemma follows by our previous estimate
that $\prob (Y(k))\leq r^{-4}$. \qed

\medskip
\noindent{\bf Proof of Lemma \ref{case3}.} Denote by $B_3$ the
required event
$$ B_3 = \Big \{ \exists \,\, (K,r){\hbox{{\rm -paths} }} \beta, \gamma \with E_s(\beta,
\gamma) \geq K_3 s \Big \} \, .$$ The analysis here is the
simplest because in this regime of $s$ such large excursions are
very rare. If $B_3$ occurs then there are at least $q={1 \over
2}K_3$ excursions of length in $[s,2s)$. Thus, if $B_3$ occurs
then there exists vertices $x_1, \ldots ,x_{2q}$ such that
\begin{enumerate}
\item There exists a perfect matching of the $x_j$ $\varphi:\{1,\ldots
,2q\}\to \{1, \ldots , 2q\}$ such that between
every pair $x_j, x_{\varphi(j)}$ there exists a simple open path
$\delta_j$ of length at least $s$.

\item For $j=1, \ldots, 2q$ there exists a simple open path
$\beta_i$ between $x_{i-1}$ and $x_i$ (denote $x_0=0$). At least
$q$ of these $\beta_j$ have length at most $2s$ and the rest have
length at most $r$.

\item The paths $\beta_j$ and $\delta_j$ are edge-disjoint (again,
there are only $q$ disjoint $\delta_j$'s).
\end{enumerate}
Denote by $X(x_1, \ldots, x_{2q})$ this event. BK inequality and
Lemma \ref{longpath} give that
$$ \prob(X) \leq 2^q q! \big ( Cs^{1-d/2} \big)^q \mathop {\sum_{A \subset \{1,\ldots, 2q\}}}_{|A|\geq q} \prod _{j\in A} \prob ( x_{j-1}
\stackrel{2s}{\leftrightarrow}x_j ) \prod_{j \not \in A} \prob (
x_{j-1} \lrr x_j ) \, .$$ We sum this over $x_1, \ldots, x_{2q}$
and use translation invariance as before to get
$$ \sum_{x_1,\ldots,x_{2q}} \prob(X) \leq \big ( C q s^{1-d/2} \big)^q \mathop {\sum_{A \subset \{1,\ldots, 2q\}}}_{|A|\geq
q} \Big ( \sum_{x \in \Z^d} \prob ( 0
\stackrel{2s}{\leftrightarrow} x) \Big)^{|A|} \Big (\sum_{x \in
\Z^d} \prob ( 0 \lrr x ) \Big )^{2q - |A|} \, ,$$ and by part
$(1)$ of Theorem \ref{exp} we get
$$ \sum_{x_1,\ldots,x_{2q}} \prob(X) \leq \big ( C q^2 r s^{2-d/2}
\big )^q \, .$$ Since $s \geq r^{1/(d/2-2.1)}$ we have that
$s^{2-d/2} \leq r^{-1 - {0.2 \over d}}$ and hence
$$ \sum_{x_1,\ldots,x_{2q}} \prob(X(x_1, \ldots, x_{2q})) \leq
\big ( C q^2 r^{-0.2/d} \big )^q \, ,$$ and choosing $K_3$ large
enough (but fixed) concludes the proof of the lemma. \qed

}


\subsection{Intrinsic metric arm exponent}\label{sub:upper} Here
we prove part (\ref{enu:Gam<}) of Theorem \ref{exp}. The proof
relies on the result of Barsky and Aizenman \cite{BA} stating that
\noindent {\em A lattice in $\R^d$ satisfying the
triangle condition satisfies, as $h\to 0$ that}
$$
\sum_{j=1}^\infty \prob(|\C(0)|=j)(1-e^{-jh})\approx h^{1/2}.
$$

\noindent This implies an estimate of $\prob(|\C(0)|)>n$. Just fix
$h=1/n$ and get \be \prob \big ( |\C(0)| > n \big ) \leq {C_1
\over n^{1/2}} \, . \ee We remark that Hara and Slade achieved a
significantly stronger estimate \cite{HaS}.

Since the event $\{|\C(0)| > n\}$ is monotone, we get \be
\label{voldec} \prob \big(|\C_G(0)|>n \big) \leq {C_1 \over
n^{1/2}} \qquad \text{ for all }G\subset E(\Z^d) \ee where
$\C_G(0)$ is the component containing $0$ in percolation on $G$
with $p=p_c(\Z^d)$ (and as in the definition of $\Gamma$, {\em
not} in the critical $p$ of $G$ itself).

\begin{proof}[{\bf Proof of part (\ref{enu:Gam<}) of Theorem
    \ref{exp}}]\label{page:II}
Let $A\geq 1$ be a large number such that
$$ 3^3 A^{2/3} + C_1 A^{2/3} \leq A \, ,$$
where $C_1$ is from (\ref{voldec}). We will now prove that
$\Gamma(r) \leq 3Ar^{-1}$. This will follow by showing
inductively that for any integer $k>0$ we have
$$ \Gamma(3^k) \leq {A \over 3^k} \, .$$
This is trivial for $k=0$ since $A \geq 1$. Assume the claim for
all $j<k$ and we prove for $k$. Let $\eps = \eps(C_1)>0$ be a
small constant to be chosen later and for any $G \subset E(\Z^d)$
write
\begin{eqnarray}
\label{indstep} \lefteqn{\prob(H(3^{k};G)) \leq}\qquad&&\nonumber\\
 &\leq &\prob \Big ( \partial
B(0,3^k;G) \neq \emptyset \and |\C_G(0)|
\leq \eps 9^{k} \Big ) + \prob \big ( |\C_G(0)| > \eps 9^{k} \big ) \nonumber \\
&\leq  &\prob \Big ( \partial B(0,3^{k};G) \neq \emptyset
\and |\C_G(0)| \leq \eps 9^{k} \Big ) + {C_1 \over \sqrt{\eps}
3^{k}} \, ,
\end{eqnarray} where the last inequality is due to
(\ref{voldec}). To estimate the first term on the right hand side
we claim that
\begin{equation}
\label{claim} \prob \Big ( \partial
B(0,3^{k};G) \neq \emptyset \and |\C_G(0)| \leq \eps 9^{k} \Big )
\leq \eps 3^{k+1}(\Gamma(3^{k-1}))^2 \, .
\end{equation}
To see this observe that if $|\C_G(0)| \leq \eps 9^{k}$ then
there must be some level $j \in [\frac{1}{3}3^k, \frac{2}{3}3^k ]$ such that
$|\partial B(0,j;G)| \leq \eps 3^{k+1}$. Denote by $j$ the first
such level. If, in addition, $\partial B(0,3^{k};G) \neq
\emptyset$ then
at least one vertex $v$ of the $\eps 3^{k+1}$ vertices of level
$j$ ``reaches level $3^{k-1}$''. Formally we do as in the proof of
lemma \ref{effres}, i.e. define $G_2$ to be $G$ with all edges
needed to calculate $B(0,j;G)$ removed and get that
\begin{equation*}
\partial B(v,3^{k-1}; G_2)\neq\emptyset
\end{equation*}
which, by the definition of $\Gamma$ (with $G_2$) has probability
$\leq \Gamma(3^{k-1})$.
Applying Markov's
inequality gives
\begin{equation*}
\prob \Big ( \partial
B(0,3^{k};G) \neq \emptyset \and |\C_G(0)| \leq \eps 9^{k} \;\Big|\; B(0,j;G)
\Big )
\leq \eps 3^{k+1}\Gamma(3^{k-1}) \, .
\end{equation*}
As in the proof of Lemma \ref{effres}, we now sum over possible values of $B(0,j;G)$ and get an extra
term of $\prob(H(0,3^{k-1};G))$ because we need to reach level
$3^{k-1}$ to begin with. We can definitely bound
$\prob(H(0,3^{k-1};G))\leq \Gamma(3^{k-1})$ and this gives the
assertion of (\ref{claim}).

We put this
into (\ref{indstep}) and get that
$$  \prob(H(3^{k};G)) \leq \eps 3^{k+1} (\Gamma(3^{k-1}))^2 + {C_1 \over \sqrt{\eps}
3^{k}} \leq \frac{ \eps 3^3 A^2 + C_1 \eps^{-1/2}}{3^k} \,
,$$ where in the last inequality we used the induction hypothesis.
Put now $\eps = A^{-4/3}$. Since the last inequality holds for any
$G\subset E(\Z^d)$ we have
$$ \Gamma(3^k) \leq {3^3 A^{2/3} + C_1 A^{2/3} \over 3^k} \leq
{A \over 3^k} \, ,$$ where the last inequality is by our choice of
$A$. This completes our inductive proof that $\Gamma(3^k) \leq A
3^{-k}$. Now, for any $r$ choose $k$ such that $3^{k-1} \leq r
< 3^k$ then we have
\begin{equation*}
 \Gamma(r) \leq \Gamma(3^{k-1}) \leq {A \over 3^{k-1}} < {3A \over r}
\, .\qedhere
\end{equation*}
\end{proof}


%

\subsection{Corresponding lower bounds.} In the following we provide the
corresponding lower bounds to the estimates of Theorem \ref{exp}.
\smallskip

\noindent{\bf Proof of part (i) of Theorem \ref{lowerexp}.} Let $x
\in \Z^d$ and write $|x|$ for the Euclidean distance of $x$ from
$0$. We estimate the quantity $\E[ d_{\Z^d_p}(0,x) \mid 0 \lr x
]$. If $0 \lr x$ then we have that $d_{\Z^d_p}(0,x)$ is no more
than the number of $y\in \Z^d$ such that the events $0\lr y$ and
$y \lr x$ occur disjointly. By the BK inequality and the two-point
function estimate (\ref{tpt}) we learn that
$$ \E [ d_{\Z^d_p}(0,x) {\bf 1}_{\{0\lr x\}} ] \leq C \sum _{y \in \Z^d}
|y|^{2-d}|x-y|^{2-d} \leq C |x|^{4-d} \, ,$$ where the last
inequality is a straightforward calculation. Hence $\E [
d_{\Z^d_p}(0,x) \mid 0 \lr x ] \leq C |x|^2$. We learn that if $x$
is such that $|x| \leq \sqrt{r/2C}$, then Markov's inequality
implies that $\prob ( d_{\Z^d_p}(0,x) \leq r \mid 0 \lr x ) \geq
1/2$. By this and (\ref{tpt}) we conclude that
$$ \E |\bcr(0,r;\Z^d)| \geq \sum _{x \, : \, |x| \leq \sqrt{r/2C}} \prob ( 0 \lr x \and
d_{\Z^d_p}(0,x) \leq r ) \geq  {1\over 2}  \sum _{x \, :\, |x|
\leq \sqrt{r/2C}} |x|^{2-d} \geq c r \, ,$$
where $c>0$ is a small constant. \qed

\smallskip
\noindent{\bf Proof of part (ii) of Theorem \ref{lowerexp}.} We
use a second moment argument. Fix some $\lambda >1$ to be chosen
later. By part (\ref{enu:EB0r}) of Theorem \ref{exp} we have that
$$ \E |\bcr (0,r)| \leq C_1 r \, ,$$
and by part (i) of Theorem \ref{lowerexp} we have
$$ \E |B(0, \lambda r)| \geq c_1 \lambda r \, .$$
Put $\lambda = 2C_1/c_1$ to get that
$$ \E |B(0,\lambda r) \setminus B(0,r)| \geq c_1 \lambda r -
C_1 r = C_1r \, .$$ We now estimate the second moment of
$|B(0,\lambda r)|$. Indeed, if $0 \lrrl x$ and $0 \lrrl y$ then
there must exist $z \in \Z^d$ such that the events $0 \lrrl z$, $z
\lrrl x$ and $z \lrrl y$ occur disjointly. Hence, the BK
inequality gives
$$ \E |B(0,\lambda r)|^2 \leq \sum_{x,y,z} \prob ( 0 \lrrl z)
\prob (z \lrrl x) \prob (z \lrrl y) = \Big [ \sum_{x \in \Z^d}
\prob (0 \lrrl x) \Big ] ^3 \leq Cr^3 \, ,$$ where the last
inequality is by part (\ref{enu:EB0r}) of Theorem \ref{exp}. The
estimate $\prob(Z>0) \geq (\E Z)^2/\E Z^2$ valid for any
non-negative random variable $Z$ yields that
$$ \prob \big ( |B(0,\lambda r) \setminus B(0,r)| > 0 \big )
\geq {C_1^2 r^2 \over Cr^3 } \geq {c \over r} \, ,$$ which
concludes our proof since the event above implies $H(r)$. \qed

\section*{Acknowledgements}
We are indebted to Yuval Peres for suggesting the use of $\Gamma$
instead of $H$ in Theorem \ref{exp}. Our original approach used a
different method of ``monotonizing'' $H$ and Peres' suggestion
has greatly simplified our proof. We would also like to thank
Chris Hoffman for valuable conversations.

The research of AN was supported in part by NSF grant \#DMS-0605166. Part of
this work was carried out while GK was a visitor at IMPA and AN was a visitor
at the Theory Group of Microsoft Research. We would like to extend our
gratefulness for the kind hospitality of both institutions, and especially
that of Vladas Sidoravicius.


\vspace{.05 in}\noindent
{\bf Gady Kozma}: \texttt{gady.kozma(at)weizmann.ac.il} \\
The Weizmann Institute of Science, \\
Rehovot POB 76100, \\
Israel.

\medskip \noindent
{\bf Asaf Nachmias}: \texttt{asafnach(at)math.berkeley.edu} \\
Department of Mathematics,
UC Berkeley\\
Berkeley, CA 94720, USA.

\end{document}